\documentclass[12pt]{article}
\usepackage{amsmath,amsfonts,amssymb,amsthm,epsfig,epstopdf,titling,url,array}
\usepackage{eqnarray}
\usepackage{enumerate}
\usepackage[a4paper]{geometry}
\usepackage{amscd}
\usepackage{centernot}
\usepackage{authblk}
\usepackage{setspace}
\usepackage{tikz-cd}
\usetikzlibrary{decorations.markings}
\tikzset{neg/.style={
        decoration={markings,
            mark= at position 0.5 with {
                \node[transform shape] (tempnode) {$\backslash$};
            }
        },
        postaction={decorate}
    }
}

\usepackage{cleveref}
\theoremstyle{plain} 
\newtheorem{thm}{ Theorem}[section]

\newtheorem{prop}[thm]{Proposition}

\newtheorem{defn}[thm]{Definition}
\newtheorem{ex}[thm]{Example}

\newtheorem* {note*}{Note}

\newtheorem{qu}{\bf Question}
\newcommand{\co} {\mathbb{C}}

\newcommand{\R} {\mathbb{R}}

\newcommand{\iy} {\infty}
\newcommand{\N} {\mathbb{N}}

\newcommand{\D} {\mathbb{D}}

\newcommand{\al}{\alpha}
\newcommand{\lm}{\lambda}

\newcommand{\eps}{\epsilon}

\newcommand{\h}{\mathcal H}

\newcommand{\bx}{\mathcal B(X)}

\newcommand{\norm}[1]{\left\Vert#1\right\Vert}
\newcommand{\abs}[1]{\left\vert#1\right\vert}
\newcommand{\set}[1]{\left\{#1\right\}}

\newcommand{\seq}[1]{\left<#1\right>}
\newcommand{\brc}[1]{\left(#1\right)}

\newcommand{\LZ}{\ell^{2}(\mathbb Z)}

\newcommand{\bo}{\mathbb{D}\backslash \left\{0\right\}}

\DeclareMathOperator{\spn}{span}
\title{\bf \Large $k$-bitransitive and compound operators on Banach spaces}
\bf
\author[1]{\bf\footnotesize Nareen Bamerni \thanks{nareen\_bamerni@yahoo.com}}
\author[2]{\bf Adem K{\i}l{\i}\c{c}man \thanks{akilicman@yahoo.com}}
\affil[1,2]{\bf Department of Mathematics, University Putra Malaysia,
43400 UPM, Serdang, Selangor, Malaysia}

\begin{document}
\date{}
\maketitle
\begin{abstract}
In this this paper, we introduce new classes of operators in complex Banach spaces, which we call $k$-bitransitive operators and compound operators to study the direct sum of diskcyclic operators. We create a set of sufficient conditions for $k$-bitransitivity and compound. We show the relation between topologically mixing operators and compound operators. Also, we extend the Godefroy-Shapiro Criterion for topologically mixing operators to compound operators. 
\end{abstract}

\section{Introduction}
A bounded linear operator $T$ on a separable Banach space $X$ is hypercyclic if there is a vector $x\in X$ such that $Orb(T,x)=\set{T^nx:n\ge 0}$ is dense in $X$, such a vector $x$ is called  hypercyclic for $T$, for more information on hypercyclic operators the reader may refer to  \cite{dynamic, chaos}. Similarly, an operator $T$ is called diskcyclic if there is a vector $x \in X$ such that the disk orbit $\D Orb(T,x)=\set{\al T^nx: \al\in \co, |\al|\le 1, n\in \N}$ is dense in $X$, such a vector $x$ is called  diskcyclic for $T$, for more details on diskcyclicity see \cite{m3,m5, cyclic}. \\

In 1982, Kitai presented in her PhD thesis some sufficient conditions for hypercyclic operators which are then called hypercyclic criterion \cite{Kitai}. Then Gethner and Shapiro \cite{Shapiro}  gave another form of this criterion. In 1987, Godefroy and  Shapiro \cite{28} created another hypercyclic criterion which is called Godefroy-Shapiro Criterion, that is a set of sufficient condition in terms of the eigenvalues of an operator to be hypercyclic.\\

It was proved that whenever the direct sum of $n$ operators is hypercyclic, then every operator is hypercyclic \cite{Kitai}. However, for the converse,  Salas constructed an
operator $T$ such that both it and its adjoint $T^*$ were hypercyclic, and so that their direct sum $T \oplus T^*$ was not. Moreover, Herrero asked in \cite{31} whether $T \oplus T$ is hypercyclic whenever is $T$. In 1999, B\'{e}s and Peris showed that an operator $T$ satisfies the hypercyclic criterion if and only if $T \oplus T$ is hypercyclic (see \cite{14}) which gives a positive answer to the Herrero's question.\\

For diskcyclic operators, Zeana proved that if the direct sum of $n$ operators is diskcyclic then every operator is diskcyclic \cite{cyclic}. However, the converse is unknown. Particularly, we have the following question:
\begin{qu}\label{104} If there are $k$ diskcyclic operators, what about their direct sum? \end{qu}

The main purpose of this paper is to give a partial answer to this question. We define and study $k$-bitransitive operators.  We determine conditions that ensure a linear operator to be $k$-bitransitive which is called $k$-bitransitive criterion. Then, we define compound operators as a general form of topologically mixing operators \cite{11}. However, we show by an example that not every compound operator is topologically mixing. Moreover, we define cross sets and junction sets to make the arguments involving  $k$-bitransitive and compound operators more transparent. Then we extend Godefroy-Shapiro Criterion \cite{28} for topologically mixing operators to compound operators. In particular,  a special case of \Cref{190} is when $p=1$ which is Godefroy-Shapiro Criterion \cite{dynamic}.  Morever, we create compound criterion which is a set of sufficient conditions for compound. Finally, We use these operators to prove that in some cases if $k$ operators are diskcyclic, then the direct sum of them is $k$-bitransitive which answer \Cref{104} for some special cases.

\section{Main results}
In this this paper, all Banach spaces are separable over the field $\co$ of complex numbers. Let $k\ge 1$ and  $T_i\in \bx$ for all $1\le i \le k$, and let $T=\bigoplus_{i=1}^k T_i: \bigoplus_{i=1}^k X \to \bigoplus_{i=1}^k X$ then we call each operator $T_i$ a component of $T$. We denote by $\D$ the closed unit disk in $\co$ and by $\N$ the set of all positive integers.

\begin{defn}\label{bi}
An operator $T\in\bx$ is called $k$-bitransitive if there exist $T_1,T_2, \ldots T_k\in \bx$ such that $T=T_1\oplus \ldots \oplus T_k$ and for any $2k$-tuples $U_1,\ldots,U_k,V_1,\ldots,V_k \subset X$ of nonempty open sets, there exist some $n\in \N$ and  $\al_1,\ldots,\al_k \in \bo$ such that $$(T_1\oplus T_2 \oplus \ldots \oplus T_k )^n(\al_1 U_1\oplus \ldots \oplus \al_k U_k)\cap (V_1\oplus \ldots \oplus V_k)\neq \phi $$
\end{defn}
It is clear from the above definition that $1$-bitransitive is identical to disk transitive which in turn  identical to diskcyclic (see \cite[Proposition 2.10]{m3}).\\

To simplify \Cref{bi}, we define the following concepts
\begin{defn}
A set $A\subset \N\times \co^p $ is called bifinite if there exist a cofinite set $K=\set{n_k:k\in \N}\subset \N$ and sequences $\seq{a^{(i)}_n}_{n\in \N}\subset \co; i=1,\ldots,p$ such that $\set{(n_k, a^{(1)}_{n_k}, \ldots, a^{(p)}_{n_k}):\mbox{ for all } k\in \N} \subseteq A$ 
\end{defn}
\begin{defn}
Let $k\ge 1$ be fixed. For all $1\le i\le k$, let $T_i$ be bounded linear operators on a Banach space $X$, and $A_i,B_i$ be non-empty subsets of $X$. Let $T=T_1\oplus \ldots \oplus T_k$, $A=A_1\oplus A_2 \oplus \ldots \oplus A_k $ and $B=B_1\oplus B_2 \oplus \ldots \oplus B_k $. The cross set from $A$ to $B$ is defined as
$$ C_T(A, B)=C(A, B)=\set{n\in \N: T^n(\al_1 A_1\oplus\ldots \oplus\al_k A_k) \cap (B_1\oplus \ldots \oplus B_k)\neq \phi}$$ for some $\al_i \in \bo; i=1,\ldots,k$
\\
\end{defn}
In the above definition, if $\al_i=\al_{i+1}$ for all $i=1,\ldots,k-1$, the return set $N(\al A,B)$ (see \cite{akin}) is equivalent to $C(A,B)$. 
\begin{defn}
Let $k\ge 1$ be fixed. For all $1\le i\le k$, let $T_i$ be bounded linear operators on a Banach space $X$, and $A_i,B_i$ be non-empty subets of $X$. Let $T=T_1\oplus \ldots \oplus T_k$, $A=A_1\oplus A_2 \oplus \ldots \oplus A_k $ and $B=B_1\oplus B_2 \oplus \ldots \oplus B_k $. The junction set from $A$ to $B$ is defined as
$ J_T(A, B)=J(A, B)=\{(n,\al_1,\ldots,\al_k)\in \N \times \D^k\backslash \set{(0,\ldots,0)}: T^n(\al_1 A_1\oplus\ldots \oplus\al_k A_k) \cap (B_1\oplus \ldots \oplus B_k)\neq \phi\}$
\end{defn}
The next proposition gives an equivalent definition to $k$-bitransitivity in terms of cross and junction sets
\begin{prop}\label{126}
Let $T=T_1\oplus \ldots \oplus T_k$. Then $T$ is $k$-bitransitive, if and only if for any $2k$-tuples of nonempty open sets $U_i,V_i \subset X$, $i=1,\ldots,k$
$$\bigcap_{i=1}^k C_{T_i}(U_i,V_i)\neq \phi$$
or if and only if for any $2k$-tuples of nonempty open sets $U_i,V_i \subset X$, there exists $\al_i\in \bo$ and $n\in \N$ such that  
$$(n,\al_i)\in J_{T_i}(U_i, V_i)$$
\end{prop}

The following theorem gives a set of sufficient conditions for $r$-bitransitivity. 

 \begin{thm}[$r$-bitransitive criterion]\label{175}
 Let $T_1, \ldots, T_r \in \bx $, suppose that there exists an increasing sequence of positive integers $\seq{n_k}$, and suppose that for all $1\le i \le r$ there exist sequences $\seq{\lm^{(i)}_{n_k}} \subset \bo$ for all $k\ge 1$, dense sets $X_i,Y_i \subset X$, and maps $S_i: Y_i \to X$  such that for all $x_i\in X_i$ and  $y_i\in Y_i$
\begin{enumerate}
\item $\bigoplus_{i=1}^r\lm^{(i)}_{n_k}T_i^{n_k}(x_1,\ldots ,x_r) \to (0,\ldots,0) $;
\item $\bigoplus_{i=1}^r\frac{1}{\lm^{(i)}_{n_k}}S_i^{n_k}(y_1,\ldots,y_r) \to (0,\ldots,0)$;
\item $\bigoplus_{i=1}^r T_i^{n_k} S_i^{n_k}(y_1,\ldots,y_r) \to (y_1,\ldots,y_r)$ .
\end{enumerate}
Then $T_1\oplus \ldots \oplus T_r$ is $r$-bitransitive.
\end{thm}
\begin{proof}
Let $U_i,V_i$ be open sets for all $1\le i \le r$ then $ \bigoplus_{i=1}^r U_i$ is open in $\bigoplus_{i=1}^r X$. Also $\bigoplus_{i=1}^r X_i$ and $\bigoplus_{i=1}^r Y_i$
are dense in $\bigoplus_{i=1}^r X$. Let $(x_1,\ldots ,x_r)\in \bigoplus_{i=1}^r U_i \cap \bigoplus_{i=1}^r X_i$ and $(y_1,\ldots ,y_r)\in \bigoplus_{i=1}^r V_i \cap \bigoplus_{i=1}^r Y_i$. By $(2)$ we have
\begin{equation}\label{173} 
(x_1,\ldots ,x_r)+ \bigoplus_{i=1}^r\frac{1}{\lm^{(i)}_{n_k}}S_i^{n_k}(y_1,\ldots,y_r) \to (x_1,\ldots ,x_r)
\end{equation}
By $(1)$ and $(3)$ we have
\begin{equation}\label{174}
\displaystyle \left( \bigoplus_{i=1}^r\lm^{(i)}_{n_k}T_i^{n_k} \right)\left( (x_1,\ldots ,x_r)+ \bigoplus_{i=1}^r\frac{1}{\lm^{(i)}_{n_k}}S_i^{n_k}(y_1,\ldots,y_r) \right) \to (y_1,\ldots,y_r)
\end{equation}
as $k \to \iy$. 
From \Cref{173} and \Cref{174}, there exist $N\in \N$ such that
$$\left( \bigoplus_{i=1}^r\lm^{(i)}_{n_k}T_i^{n_k} \right) \left( \bigoplus_{i=1}^r U_i \right) \cap \bigoplus_{i=1}^r V_i \neq \phi \mbox{ for all } k\ge N$$
which is equivalent to 
$$(T_1\oplus \ldots \oplus T_r)^{n_k}(\lm^{(1)}_{n_k}U_1 \oplus \ldots \oplus \lm^{(r)}_{n_k}U_r )\cap (V_1 \oplus \ldots  \oplus V_r)\neq \phi \mbox{ for all } k\ge N$$
that is 
$$\bigcap_{i=1}^r C_{T_i}(U_i,V_i)\neq \phi$$
By \Cref{126}, $T_1\oplus \ldots \oplus T_r$ is $r$-bitransitive.  
\end{proof}

The following proposition gives another criterion for $r$-bitransitive operators without the need of scalar sequences.
\begin{prop}\label{176}
 Let $T_1, \ldots, T_r \in \bx $, suppose that there exists an increasing sequence of positive integers $\seq{n_k}$, and suppose that for all $1\le i \le r$ there exist dense sets $X_i,Y_i \subset X$, and maps $S_i: Y_i \to X$  such that for all $x_i\in X_i$ and  $y_i\in Y_i$
\begin{enumerate}
\item $\bigoplus_{i=1}^r T_i^{n_k}(x_1,\ldots ,x_r) \bigoplus_{i=1}^r S_i^{n_k}(y_1,\ldots,y_r)\to (0,\ldots,0) $;  \label{a22}
\item $\bigoplus_{i=1}^r S_i^{n_k}(y_1,\ldots,y_r) \to (0,\ldots,0)$; \label{b22}
\item $\bigoplus_{i=1}^r T_i^{n_k} S_i^{n_k}(y_1,\ldots,y_r) \to (y_1,\ldots,y_r)$ . \label{c22}
\end{enumerate}
Then $T_1\oplus \ldots \oplus T_r$ is $r$-bitransitive.
\end{prop}
The proof of the above proposition follows from the next proposition and \Cref{175}
\begin{prop}\label{177}
Both $r$-bitransitive criteria are equivalent.
\end{prop}
\begin{proof}
Let the hypothesis of \Cref{175} are given. By $(1)$ and $(2)$ in \Cref{175}, we get $(1)$ in \Cref{176}. Since $\seq{\frac{1}{\lm^{(i)}_{n_k}}} \not\to 0$ for all $1\le i \le r$, then $(2)$ in \Cref{175} implies that $\bigoplus_{i=1}^r S_i^{n_k}(y_1,\ldots,y_r) \to (0,\ldots,0)$; i.e, $(2)$ in \Cref{176} holds. It follows that the hypothesis of \Cref{176} hold true.\\
Conversely, suppose that the hypothesis of \Cref{176} are given. Then there exist a large $p\in \N$ and a small positive integer $\eps$ such that 
\begin{equation}\label{179}
\left\|\bigoplus_{i=1}^r T_i^{n_k}(x_1,\ldots ,x_r) \bigoplus_{i=1}^r S_i^{n_k}(y_1,\ldots,y_r)\right\|\le \eps^2 
\end{equation}
 and 
\begin{equation}\label{181}
\left\|\bigoplus_{i=1}^r S_i^{n_k}(y_1,\ldots,y_r) \right\|\le \eps 
\end{equation}
whenever $k\ge p$. \Cref{181} implies that for all  $1\le i \le r$ and  $k\ge p$  
$$ \left\| S_i^{n_k}y_i \right\|\le \eps $$
Let $\lm^{(i)}_{n_k}=\frac{1}{\eps}S_i^{n_k}y_i$ which implies that $ \left| \lm^{(i)}_{n_k}\right|\le 1$ and $ \left\| \frac{1}{\lm^{(i)}_{n_k}} S_i^{n_k}y_i \right\| = \eps $. It follows that
$$ \frac{1}{\lm^{(i)}_{n_k}} S_i^{n_k}y_i \to 0 \mbox{ for all } 1\le i \le r \mbox{ and } k\ge p$$
and so
\begin{equation}\label{178}
 \bigoplus_{i=1}^r\frac{1}{\lm^{(i)}_{n_k}} S_i^{n_k}(y_1,\ldots,y_r) \to 0 \mbox{ as } k\to \iy
\end{equation}
Furthermore, by \Cref{179}
$$\left\| \left( T_1^{n_k}x_1 S_1^{n_k}y_1,\ldots,  T_r^{n_k}x_r S_r^{n_k}y_r \right) \right\|= \left\| T_1^{n_k}x_1 S_1^{n_k}y_1 \right\|+\ldots+ \left\| T_r^{n_k}x_r S_r^{n_k}y_r \right\|  \le \eps^2 $$ 
Therefore,
$\left\| T_i^{n_k}x_i S_i^{n_k}y_i \right\|= \left\| T_i^{n_k}x_i \right\| \left|\lm^{(i)}_{n_k}\right| \eps\le \eps^2 $ for all  $1\le i \le r$ and  $k\ge p$. Then
$$\sum_{i=1}^r \left\|\lm^{(i)}_{n_k} T_i^{n_k}x_i \right\| \le r\eps $$
It follows that 
$$\left\| \left(\lm^{(1)}_{n_k} T_1^{n_k}x_1 ,\ldots, \lm^{(r)}_{n_k} T_r^{n_k}x_r  \right) \right\| \le r\eps $$
so
$$\left\|\bigoplus_{i=1}^r \lm^{(i)}_{n_k} T_i^{n_k}(x_1,\ldots ,x_r) \right\|\le r\eps $$
If we let $\eps$ small enough we get 
\begin{equation}\label{180}
\bigoplus_{i=1}^r \lm^{(i)}_{n_k} T_i^{n_k}(x_1,\ldots ,x_r)\to (0,\ldots,0) \mbox{ as } k \to \iy
\end{equation}
The proof follows from \Cref{178} and \Cref{180}
\end{proof}

Now to answer \Cref{104}, we will define another class of operators which is called compound operators. 
\begin{defn}
An operator $T\in\bx$ is called compound if for any nonempty open sets $U,V$, there exist some $N\in \N$ and a sequence  $\seq{\al_n} \subset \bo$ such that $$T ^n(\al_n U)\cap V\neq \phi $$ for all $n\ge N$
\end{defn}
\begin{prop}\label{163}
An operator $T$ is compound, if and only if for any two nonempty open sets $U, V \subset \h$, 
$$ J(U,V) \mbox{ is bifinite }$$
\end{prop}
It is clear that every compound operator is disk transitive. A special case of compound operator is when $\al_n=1$ for all $n\ge N$, and it is called topologically mixing operators (see \cite{11}). Therefore every topologically mixing operator is compound. However, not every compound operator is topologically mixing as shown in the following example
\begin{ex}
Let $T:\LZ\to\LZ$ be the bilateral forward weighted shift with
the weight sequence
\begin{equation*}
w_n=
\begin{cases} R_1 & \text{if $n \geq 0$,}
\\
R_2 &\text{if $n< 0$.}
\end{cases}
\end{equation*}
where $R_1, R_2\in \R^+; R_1<R_2$. Then $T$ is compound not topologically mixing.
\end{ex}
\begin{proof}
Let $U$ and $V$ be two open sets. Since $T$ satisfies diskcyclic criterion with respect to the sequence $\seq{n}_{n\in\N}$ (see \cite[Example 2.20.]{m3}), then there exist two dense set $D_1$ and $D_2$ such that 
\begin{enumerate}
\item \label{8a} For each $y\in V$, $\lim_{n\to \iy}\norm{B^ny}\to 0$ (where $B$ is the bilateral backward weighted shift)
\item \label{8b} For all $x \in D_1$ and $y\in D_2$, $\lim_{n\to \iy}\norm{T^{n}x}\norm{B^ny}\to 0$  ;
\item \label{8c} For each $y\in D_2$, $T^{n}B^ny =y$ ;
\end{enumerate}
Let $x\in U\cap D_1$, $y\in V\cap D_2$ and $\lm_n=\frac{\norm{B^ny}^\frac{1}{2}}{\norm{T^nx}^{\frac{1}{2}}}$. For an arbitrary large $N\in \N$, suppose that $z=x+\lm_N^{-1}B^Ny$, then, $\norm{z-x}=\brc{\norm{T^Nx}\norm{B^Ny}}^\frac{1}{2}\to 0$ by (\Cref{8a}) and thus $z\in U$. Also, $\lm_NT^Nz=\lm_NT^Nx+T^NB^Ny=\lm_NT^Nx+y$ by (\Cref{8c}). Then $\norm{\lm_NT^Nz-y}=\lm_N\norm{T^Nx}=\brc{\norm{T^{N}x}\norm{B^Ny}}^\frac{1}{2}\to 0$ by (\Cref{8b}) and so $\lm_NT^Nz\in V$. Since $\lim_{n\to \iy}\norm{T^nx}\to \iy$ and $\lim_{n\to \iy}\norm{B^nx}\to 0$ then $\lim_{n\to \iy} \lm_n\to 0$. Therefore $\lm_NT^NU\cap V\neq \phi$. Since $N$ is arbitrary, we can assume that $J(U,V)=\set{(\lm_n,n):n\ge N}$ and hence $T$ is compound. Since $T$ is not topological transitive (see \cite[Example 2.20.]{m3}), then $T$ can not be topologically mixing.
\end{proof}
The following theorem extends the Godefroy-Shapiro Criterion \cite{28} for topologically mixing operators to compound operators.

\begin{thm} \label{190}
Let $T\in \bx$. If there exists a $p\ge 1$ such that
$$A=\spn \set{x\in X: Tx=\al x \mbox { for some } \al \in \co; |\al|<p};$$
$$B=\spn \set{y\in X: Ty=\lm y \mbox { for some } \lm \in \co; |\lm|>p};$$
are dense in $X$, then $T$ is compound.
\end{thm}
\begin{proof}
Let $U$ and $V$ be nonempty open sets in $X$. Since $A$ and $B$ are dense, then there exist $x\in A \cap U$ and $y\in B\cap V$. Then $x=\sum_{i=1}^k a_i x_i$ and $y=\sum_{i=1}^k b_i y_i$ where $a_i,b_i\in \co$ for all $1\le i \le k$ . Also, $Tx_i=\al_i x_i$ and  $Ty_i=\lm_i y_i$ where $|\al_i|<p$ and $|\lm_i|>p$ for all $1\le i \le k$. Let $c\in \co$ be a scalar such that  $p\le |c| < |\lm_i|$ for all $1\le i \le k$, and let  
$$z_n=\sum_{i=1}^k b_i (\frac{c}{\lm_i})^n y_i \mbox{ for all } n\ge 0$$
Then 
$$\frac{1}{c^n} T^n x=\sum_{i=1}^k a_i (\frac{\al_i}{c}) ^n x_i\to 0 \mbox{ and } z_n\to 0 \mbox{ as } n\to \iy$$
and $\frac{1}{c^n} T^n z_n=y$ for all $n\ge 0$. It follows that there is a positive integer $r$ such that for all $n\ge r$
$$x+z_n\in U \quad \mbox{ and } \quad \frac{1}{c^n}T^n(x+z_n)=\frac{1}{c^n}T^nx+\frac{1}{c^n}T^nz_n\in V \mbox{ for all } n\ge r$$
Therfore, $\frac{1}{c^n}T^n U\cap V\neq \phi $ for all $n\ge r$. It follows that $T$ is compound.
\end{proof}
Note that in the above theorem, if $p=1$, then it will be a Godefroy-Shapiro criterion for topologically mixing operators. \\

The following theorem gives another criterion for compound operators.

\begin{thm}\label{c1}
Let $T \in \bx $, suppose that there exist a sequence $\seq{\lm_{n}} \subset \co\backslash\set{0}$ such that $\abs{\lm_{n}}\le 1$ for all $n\in \N$, two dense sets $D_1,D_2 \subset \h$ and a sequence of maps $S_{n} : Y \to \h$ such that as $n\to \iy$:
\begin{enumerate}
\item $\lm_{n}T^{n}x \to 0$ for all $x \in D_1$;
\item $\frac{1}{\lm_{n}}S_{n}y \to 0$ for all $y \in D_2$;
\item $T^{n} S_{n}y \to y$ for all $y \in D_2$.
\end{enumerate}
Then  $T$ is compound and it is called compound with respect to the sequence $\seq{\lm_{n}}$.
\end{thm}
\begin{proof}
Let $U,V$ be non empty open sets, and let $x\in U\cap D_1$ and $y\in V\cap D_2$. Then $x+\frac{1}{\lm_{n}}S_{n}y \to x\in U$ as $n\to \iy$ and $\lm_{n}T^{n} (x+\frac{1}{\lm_{n}}S_{n}y)=\lm_{n}T^{n}x+y\to y \in V $. Thus there exists a large positive integer $N$ such that $T^n\lm_{n} U \cap V \neq\phi $ for all $n\ge N$. It follows that $T$ is compound.
\end{proof}
The following theorem gives another criterion for compound operators. 
\begin{thm}\label{c2}
Let $T \in \bx $. If there exist two dense sets $D_1,D_2 \subset \h$ and a sequence of maps $S_{n} : Y \to \h$ such that:
\begin{enumerate}
\item $(T^{n}x)( S_{n}y)\to 0$ for all $x \in D_1$ and $y \in D_2$ ;
\item $S_{n}y \to 0$ for all $y \in D_2$;
\item $T^{n} S_{n}y \to y$ for all $y \in D_2$.
\end{enumerate}
Then  $T$ is compound.
\end{thm}
The proof of the above theorem can be followed by showing that both compound criteria in \Cref{c1} and \Cref{c2} are equivalent by using the same lines in \Cref{177}.\\

Now, the following two theorems give a partial answer to \Cref{104}.
\begin{thm}\label{220}
Let $T=T_1\oplus \ldots \oplus T_k$. If every component of $T$ is disk transitive and at least $(k-1)$ of them is compound then $T$ is $k$-bitransitive.
\end{thm}
\begin{proof}
We will prove the case $k=2$ and the other cases are same. Let $T_1$ and $T_2$ be disk transitive operators and let $T_1$ be compound without loss of generality. Let $U_1,U_2,V_1,V_2$ be nonempty open sets, then there exist $N_1, N_2\in \N$, $\al_1\in \bo$ and a sequence $\seq{\beta_n}_{n\in N} \subset \bo$ such that $T_2^{N_1}\al_1 U_1\cap U_2\neq \phi$ and $T_1^{n}\beta_n V_1\cap V_2\neq \phi$ for all $n\ge N_2$. Since $J_{T_2}(U_1,U_2)$ is infinite then there exist $N\in \N, N\ge N_2 $ and $\al\in \bo$ such that $T_2^{N}\al U_1\cap U_2\neq \phi$ and $T_1^{N}\beta_N V_1\cap V_2\neq \phi$. It follows that $C_{T_1} (U_1, U_2) \cap C_{T_2} (V_1, V_2)\neq \phi$ and hence $T$ is $2$-bitransitive. 
\end{proof}

\begin{thm}
If $r$ operators satisfy diskcyclic criterion for the same sequence $\seq{n_k}_{n\in N}$, then their direct sum is an $r$-bitransitive operator.
\end{thm}  
\begin{proof}
Let $T_i$ satisfies diskcyclic criterion with respect to the sequence $\seq{n_k}_{k\in \N}$ for all $1\le i\le r$. Then for all $1\le i\le r$, there exist $2k$ dense sets $D_i,D'_{i}$, $r$ maps $S_i$ such that 
\begin{eqnarray}
T_i^{n_k}x_i S_i^{n_k}y_i\to 0 \label{a20}\\
S_i^{n_k}y_i\to 0 \label{b20}\\
T_i^{n_k}S_i^{n_k}y_i\to y_i \label{c20}
\end{eqnarray}
as $k\to \iy$. By equation (\ref{a20}), we get 
\begin{eqnarray}
(T_1^{n_k}x_1 S_1^{n_k}y_1, \ldots, T_r^{n_k}x_r S_r^{n_k}y_r)\to (0,\ldots,0) \nonumber\\
(T_1^{n_k}x_1 \ldots, T_r^{n_k}x_r )(S_1^{n_k}y_1, \ldots, S_r^{n_k}y_r)\to (0,\ldots,0) \nonumber\\
\bigoplus_{i=1}^r T_i^{n_k} (x_1,\ldots, x_r) \bigoplus_{i=1}^r S_i^{n_k} (y_1,\ldots, y_r) \to (0,\ldots,0) \label{a21}
\end{eqnarray}
By equation (\ref{a21}), condition (\ref{a22}) of  \Cref{176} holds. \\
Also by equation (\ref{b20}), we get 
\begin{eqnarray}
(S_1^{n_k}y_1, \ldots,S_r^{n_k}y_r)\to (0,\ldots,0) \nonumber\\
\bigoplus_{i=1}^r S_i^{n_k} (y_1,\ldots, y_r) \to (0,\ldots,0) \label{a23}
\end{eqnarray}
It follows by equation (\ref{a23}), that condition (\ref{b22}) of  \Cref{176} holds. \\
Finally, by equation (\ref{c20}), we get 
\begin{eqnarray}
(T_1^{n_k} S_1^{n_k}y_1, \ldots, T_r^{n_k}S_r^{n_k}y_r)\to (y_1,\ldots,y_r) \nonumber\\
\bigoplus_{i=1}^r T_i^{n_k}S_i^{n_k} (y_1,\ldots, y_r) \to (y_1,\ldots,y_r) \label{a24}
\end{eqnarray}
It follows by equation (\ref{a24}), that condition (\ref{c22}) of  \Cref{176} holds. By \Cref{176}, $\bigoplus_{i=1}^r T_i$ is $r$-bitransitive.
\end{proof}

\newenvironment{conclusion}%
    {\null\begin{center}%
    \bfseries Conclusion\end{center}}%
    
\begin{conclusion}
We define new classes of operators on Banach spaces which are called $k$-bitransitive operators and compound operators. We create some criteria for them, and we extend the Godefroy-Shapiro Criterion for topologically mixing operators to compound operators. We use these operators to show that the direct sum of $k$-diskcyclic operators is $k$-bitransitive for some special cases. However, it seems that the \Cref{104} remains open. Therfore, one may answer a special case of that question, that is, when all $k$ diskcyclic operators are identical. Particularly. 
\begin{center}{\it
If $T$ is a diskcyclic operator. What about the $k$-fold direct sum of $T$?}
\end{center}
\end{conclusion}

\end{document}